\documentclass[11pt]{amsart}
\usepackage[totalwidth=14cm,totalheight=21cm]{geometry}
\usepackage{mathrsfs}

\usepackage{amssymb}

\usepackage{tipa}

\usepackage{upgreek}

\newtheorem{lemma}{Lemma}[section]

\newtheorem{remark}{Remark}[section]

\newtheorem{theorem}{Theorem}[section]

\begin{document}
\title{Remarks on the critical catenoid and an\\
 overdetermined eigenvalue problem in $\Bbb S^2$}
%\date{\today}

\author{Rabah SOUAM}
\address{Institut de Math{\'e}matiques de Jussieu-Paris Rive Gauche,  CNRS  UMR 7586, Universit{\'e} Paris Cit{\'e}, B{\^a}timent SoFe Germain
 Case 7012,  75205 Paris Cedex 13, France.}
\email{rabah.souam@imj-prg.fr}

\newtheorem{proposition}{Proposition}[section]

\begin{abstract} It was recently proved that an embedded minimal annulus with free boundary in the unit ball $\Bbb B^3$ of the Euclidean space is, up to a rotation, the critical catenoid. We prove that an  immersed, possibly branched, free boundary minimal annulus in $\Bbb B^3$ with embedded boundary components is necessarily free of branch points and is embedded. This shows the uniqueness of the critical catenoid holds under these weaker hypotheses. We then apply these results to prove that if $\Omega$ is a sufficiently regular  annulus in $\Bbb S^2$ for which there exists a function satisfying the overdetermined eigenvalue problem
\begin{equation*}
\begin{cases} \Delta u+ 2u = 0 \quad \text{in} \quad \Omega\\
\qquad \quad u=0  \quad \text{on}\quad\partial\Omega\\
\quad\,\,\,\, |\nabla u|=1 \quad \text{on}\quad \partial\Omega.
\end{cases}
\end{equation*}
 then $\Omega$ is, up to rotation,  a rotational  annulus with equatorial symmetry.
\end{abstract}

\maketitle

%\keywords{constant mean curvature, stability}

%\subjclass{}

\section{Introduction }
\medskip
Let $\Bbb B^3$ denote the closed unit ball centered at the origin in the Euclidean space $\Bbb R^3$ and  $ \Bbb S^2=\partial\Bbb  B^3$ its boundary.  Given   a  smooth compact and connected surface  $\Sigma$ with non empty boundary 
$\partial \Sigma,$  an immersion   $F: \Sigma  \longrightarrow \Bbb B^3,$ which is smooth up to the boundary,  is said to be a minimal immersion with free boundary if it is minimal, 
$F(\partial \Sigma) \subset \partial \Bbb B^3$ and $F(\Sigma) $ meets orthogonally $\partial \Bbb B^3$ along $F(\partial \Sigma).$ 

The simplest examples of such surfaces are the equatorial disk and the so called critical catenoid. The critical catenoid is the surface of revolution obtained by rotating about the $z-$axis the graph of the function
$$ a \cosh \frac{z}{a},$$
over the interval 
$$\left[ -\frac{1}{\cosh L}, \frac{1}{\cosh L}\right],$$
where $L$ is the unique positive solution to the equation $L=\coth L,$ and $a=1/ (L\cosh L).$

\medskip

Several works have been devoted during the last years to study the existence and uniqueness of minimal surfaces with free boundary in $\Bbb B^3.$ Many existence results, for various topological types of surfaces, have been obtained using different methods. On the contrary, uniqueness results proved to be more difficult to obtain. In his pioneering work, Nitsche \cite{N} proved that the equatorial disk is the only such surface having the topology of the disk and claimed without proof that the critical catenoid is the only {\it embedded } example having the topology of an annulus. This became a conjecture in this topic and was settled very recently \cite{CG, PZ}.

\medskip
In this paper, we will first prove, in section \ref{sec:embed},  that an immersed, possibly branched, minimal annulus with free boundary in $\Bbb R^3$ is necessarily free of branch points and is embedded provided each of its boundary components is embedded. Thus the uniqueness of the critical catenoid holds under this weaker hypothesis, showing a stronger rigidity.  
It is noteworthy  that Fernandez-Hauswirth-Mira \cite{FHM},  and Kapouleas-McGrath \cite{KM} proved the existence of immersed, non embedded, minimal annuli with free boundary in $\Bbb B^3.$ 
\medskip

As a second result, we will apply the above uniqueness result to prove, in section \ref{sec:overdetermined},  a uniqueness result  for annuli in the unit sphere $\Bbb S^2$ which have solutions for an overdetermined eigenvalue problem. This originated from our  paper \cite{S} where we observed the 
strong link, through the support function,  between free boundary minimal surface of genus zero in $\Bbb B^3$  and domains $\Omega$ in $\Bbb S^2$ possessing a smooth $u$ verifying

\begin{equation}\label{system}
\begin{cases} \Delta u+ 2u = 0 \quad \text{in} \quad \Omega\\
\qquad \quad u=0  \quad \text{on}\quad\partial\Omega\\
\quad\,\,\,\, |\nabla u|=1 \quad \text{on}\quad \partial\Omega.
\end{cases}
\end{equation}

\medskip

Indeed, the existence of  a function $u$ satisfying  the previous conditions gives rise, through the map 
$X_u:= \nabla u + u\,\text{Id},$ to a branched minimal immersion with free boundary in $\Bbb B^3.$ The simplest examples of 
domains having the above property are  the hemispheres.  As a consequence of this correspondence, we obtained in \cite{S}, using (a slight extension of)  Nitsche's uniqueness result,  that a simply connected domain $\Omega$ in $\Bbb S^2$ satisfying the above hypothesis is necessarily a hemisphere. Another simple example of such a domain is the image by the Gauss map of the critical catenoid.  We conjectured 
in \cite{S} that these are, up to rotations,  the only  domains having this property. This was recently disproved by Hines, Colesar and McGrath \cite{HCM} who constructed counterexamples given by domains having $m$ boundary components, for any large enough $m.$ Nevertheless, in the same spirit we used Nitsche's theorem, 
using  the uniqueness of the critical catenoid, we are able to prove that   an annulus in $\Bbb S^2$ supporting  a solution to the overdetermined eigenvalue problem (\ref{system}) is necessarily, up to a rotation, an annulus of revolution with an equatorial symmetry, namely the image by the Gauss map  of the critical catenoid.

\medskip
 As another application, we also show that the critical catenoid is, up to rotations,  the only immersed free boundary minimal surface in $\Bbb B^3$ with injective Gauss map.
\bigskip

The author  started to prepare this preprint when the paper by Chodosh-Giaconna  \cite{CG} appeared on arXiv and  later became aware of the paper by Parise-Zhu \cite{PZ} also  establishing the uniqueness of the critical catenoid and proving a theorem similar to Theorem \ref{annulus} below,  under the additional hypothesis that $u>0$ in $\Omega.$ 
He  is grateful to Eric Toubiana for a careful reading of a first version of this paper and for useful remarks. 

\bigskip

\section{A Sufficient condition for embeddedness of minimal annuli with free boundary}\label{sec:embed}

\medskip
We start recalling some basic facts about branched minimal surfaces in $\Bbb R^3.$  A branched (or generalized) minimal surface in $\Bbb R^3$ is the data of a Riemann surface $\mathcal R,$ possibly with boundary, and a non-constant  harmonic mapping
$$F : \mathcal R \longrightarrow \Bbb R^3$$
which is almost conformal, that is,
$$\langle F_z, F_z\rangle =0$$ 
where $z$ is any  local conformal coordinate on $\mathcal R.$ 

The  zeros of $F_z$ are called branch points. Interior branch points are isolated  and $F$ is a  conformal minimal  immersion on $\mathcal R\setminus \{F_z=0\}. $  If   $F(\mathcal  R)$ is not contained in a horizontal plane, then $F$ can be expressed, up to translation by $F(p_0), \, p_0\in \mathcal R,$  by its Weierstrass representation   as follows
\begin{equation}\label{weierstrass}
F(p)=\Re \int_{p_0}^p \left(\frac{1}{2}(1-g^2), \frac{i}{2}(1+g^2), g\right)\omega,\quad p\in\mathcal R,
\end{equation}
where $g$ is a meromorphic function and $\omega$ a holomorphic one form on $\mathcal R$  satisfying the condition that whenever $g$ has a pole of order $m,$ $\omega$ has a zero of order $\geq 2m.$
 The branch points of $F$ are those where $g$ has a pole of order $m$ (including the case $m=0$) and $\omega$  a zero of order $>2m.$ Denoting by $\sigma$ the stereographic projection of $\Bbb S^2$ to the equatorial plane from the north pole, the map $N=\sigma^{-1}\circ g$ is smooth on $\mathcal R$ and is a unit normal field to the immersion, away from the branch points, that is, its Gauss map. 
 
Fix a conformal coordinate $z$ on $\mathcal R$ and write $\omega(z)=f(z)dz,$ the induced metric on $\mathcal R,$  which is singular at the branch points, is given by :
 $$ds^2 = \frac{1}{4}|f|^2(1+|g|^2)^2|dz|^2 $$
 and the curvature on the regular set  :
 \begin{equation}\label{curvature}
 K=-\left(\frac{4|g^\prime|}{|f|(1+|g|^2)^2}\right)^2.
 \end{equation}

\medskip

%%%%%%%%%%%%%%%
%%%%%%%%%%%%%%%
%%%%%%%%%%%%%%%

Before stating our first result, we  observe the following fact that we will use later.

\begin{lemma}\label{sign}
Let $\Omega\subsetneq \Bbb S^2$ be a smooth connected domain in $\Bbb S^2.$  Suppose $2$ is an eigenvalue of  the Laplacian with Dirichlet boundary condition on $\Omega,$   that is, there exists a  non identically zero  function $u\in\mathcal C^2(\Omega)\cap \mathcal C^0(\overline\Omega)$
verifying 
\begin{align*} \Delta u +2u &=0\qquad \text{in}\quad \Omega\\
u&=0 \qquad \text{on}\quad \partial\Omega.
\end{align*}
Then $2$ is the first eigenvalue of  the Laplacian with Dirichlet boundary condition on $\Omega$ and so $u$ does not vanish in the interior of $\Omega.$ 
\end{lemma}

\begin{proof} Let $\lambda_1^D(\Omega)<\lambda_2^D(\Omega)$ be the first and the second eigenvalues of the Laplacian 
with Dirichlet boundary condition on $\Omega.$ To prove the lemma  it is enough to show that $\lambda_2^D(\Omega)>2.$ This will follow from the monotonicity of the eigenvalues of the Laplacian with respect to the inclusion. More precisely, recall that the first non zero eigenvalue of the Laplacian on $\Bbb S^2$ 
is equal to 2 and by the minimax principle 
\begin{equation}\label{minimax} \inf_{{\varphi \in H^1(\Bbb S^2), \int_{\Bbb S^2}\varphi =0}}\,\,\frac{\int_{\Bbb S^2} |\nabla \varphi|^2}{\int_{\Bbb S^2}\varphi^2} =2.
\end{equation} 
 Moreover, if $\varphi\in H^1(\Bbb S^2)$ achieves the equality then $\varphi$ is an eigenfunction associated to the the eigenvalue 2, namely the restriction to $\Bbb S^2$ of a linear function on $\Bbb R^3.$ 
 
 Consider an eigenfunction $\varphi_1$ (resp. $\varphi_2$) associated to $\lambda^D_1(\Omega)$ (resp. to $\lambda^D_2(\Omega))$  with 
  $\int_{\Omega} \varphi_1^2=1$ (resp. $\int_{\Omega} \varphi_2^2=1$). Extending $\varphi_1$ and $\varphi_2$ by zero outside $\Omega,$ and keeping the same notation,  we get $\varphi_1, \varphi_2 \in H^1(\Bbb S^2).$ Now, we  can find  $a_1,a_2\in \Bbb R,$ not both equal to 0, so that 
 $$\int_{\Bbb S^2} a_1\varphi_1+a_2\varphi_2=0.$$
   By (\ref{minimax}) we have 
 $$2\leq \frac{a_1^2\int_{\Omega}|\nabla\varphi_1|^2 + a_2^2\int_{\Omega}|\nabla\varphi_2|^2}{a_1^2+a_2^2}= \frac{a_1^2\lambda^D_1(\Omega)+ a_2^2\lambda_2^D(\Omega)}{a_1^2+a_2^2}\leq \lambda_2^D(\Omega)$$
 Moreover, if we had $\lambda_2^D(\Omega)=2,$ then necessarily $a_1=0$ and $\varphi_2$ would achieve the equality in (\ref{minimax}) and thus would be an eigenfunction  of the Laplacian on $\Bbb S^2$ associated to the eigenvalue 2,  which is absurd since $\varphi_2$ vanishes on the complement of $\Omega.$ This shows that
  $\lambda_2^D(\Omega) >2.$
 \end{proof}

We will now prove our first  result as announced in the introduction.

\begin{theorem}\label{embedding} Let $\Sigma$  be an immersed minimal annulus, possibly with branch points,  with free boundary in $\Bbb B^3.$   Then $\Sigma$ is free of branch points. Moreover if each boundary component of $\Sigma$ is embedded then $\Sigma$ is embedded  and thus, by the uniqueness of the critical catenoid, $\Sigma$ is, up to a rotation,  the critical catenoid.
\end{theorem}
  
  \begin{proof}
  Recall that  any Riemann surface having the topology of a closed annulus is conformal to a ring 
 $\mathcal A=\{ z\in \Bbb C, r\leq |z|\leq 1\} $ in the complex plane $\Bbb C,$ for some $0<r<1.$  So the surface $\Sigma$ is 
  given by an  almost conformal non-constant harmonic map 
 $$F: \mathcal A\longrightarrow \Bbb B^3$$
 verifying $$F(\partial\mathcal A)\subset \partial\Bbb B^3$$
 and such that $F(\mathcal A)$ intersects $\partial\Bbb B^3$ orthogonally away from the branch points. 
 As, obviously, $F(\mathcal A)$ is not contained in a plane, $F$ is given,  in terms of its  Weierstrass data $(g,f),$ as in (\ref{weierstrass}) with 
 $\omega(z)=f(z) dz.$
 
 We will first prove that $F$ is free of branch points, following an argument we used in \cite{S}.  We let $\mathcal S$ be the  set of branch points of $F.$ Denote by II the {\it second fundamental form} of $F.$  It is a priori defined on the regular set  $\mathcal A\setminus\mathcal S,$ but since the Gauss map $N$ extends smoothly to all of $\mathcal A,$ so does  II  by setting
 $$\text{II}=-\langle dN, dF\rangle \quad \text{on}\quad \mathcal A.$$
 Note that II vanishes on $\mathcal S.$ 
 The  Hopf-function $h$  of $F$ is the smooth function defined on $\mathcal A $ by:
$$ h(z)= \frac{1}{2}(\text{II}(\partial_x, \partial_x)-\text{II}(\partial_y, \partial_y))- i\, \text{II}(\partial_x, \partial_y). $$

It is a well known fact that minimality  implies $h$ is holomorphic on $\mathcal A\setminus\mathcal S.$    The interior branch points being isolated, $h$ is holomorphic on int$(\mathcal A).$ Also, $\partial \mathcal A\setminus \mathcal S$ is an open subset of $\partial\mathcal A$ and, by the free boundary condition, each of its components is a  line of curvature. Therefore 
$$\text{II} (x\partial_x+y\partial_y, -y\partial_x+x\partial_y)=0\qquad \text{on} \quad \partial \mathcal A\setminus \mathcal S.$$
Since $\text{II}$ vanishes on $\mathcal S,$ we have 
$$\text{II}(x\partial_x+y\partial_y, -y\partial_x+x\partial_y)=0\qquad \text{on} \quad \partial \mathcal A.$$

This can be rewritten as follows:
$$ \Im  (z^2h(z)) = 0 \qquad \text{on}\quad \partial \mathcal A.$$

By the maximum principle,  the harmonic  function  $ \Im  (z^2h(z))$ is identically zero on $\mathcal A$ and hence $z^2h(z)$ is a real constant:  there exists $a\in \Bbb R$ so that 
$$ h(z)= \frac{a}{z^2} , \qquad z\in \mathcal A.$$

Note that $a\neq 0$ since, otherwise, $h$ would be identically zero and thus the immersion would be totally umbilical, that is, a spherical cap or a totally geodesic disk, which is excluded.

In terms of the Weierstrass data $(g,f)$ of the immersion $F$, we have (cf. \cite{O}): 
$$h= -fg^\prime$$
Therefore 
$$f g^\prime(z) = -\frac{a}{ z^2}\qquad z\in\mathcal A$$

This shows that $F$ has no branch points, that is, $F$ is an immersion on all of $\mathcal A.$  This also shows that $g$ can have only simple poles and  $g^\prime$ has no zeros. It follows that 
$g$ and hence its Gauss map $N:\mathcal A\longrightarrow \Bbb S^2$  is a local diffeomorphism and that the Gauss curvature $K$ of the surface never vanishes (cf. (\ref{curvature})). 

We suppose now that $F$ is an embedding when restricted to each boundary component of $\mathcal A.$ We will first  prove  the Gauss map $N$ is a diffeomorphsim onto its image $ N(\mathcal A).$  
Let  $\Gamma$ be a component of $\partial \mathcal A,\, \gamma$ a parameterization by arclength of $\Gamma$ and set 
 $\overline\gamma=F\circ \gamma.$  
 By the free boundary condition, $N$ is a unit normal in $\Bbb S^2$ to the curve  $F(\Gamma)$ and so
  $\langle \ddot {\overline\gamma}, N\rangle$ coincides, possibly up  to sign, with the geodesic curvature $\kappa_g$ of $F(\Gamma)$ in $\Bbb S^2.$ 
We know $\gamma$
is a line of curvature, so $|\kappa_g|= |\langle \ddot {\overline\gamma}, N\rangle| = |\text{II} (\dot\gamma,\dot\gamma)|= \sqrt{-K}\neq 0.$  Therefore $F(\Gamma)$ is a closed locally convex curve in $\Bbb S^2$ and as, by assumption, it is embedded, we conclude that 
it is contained in an open hemisphere and is globally convex (cf. \cite{F}).

Set $\partial\mathcal A=\Gamma_1\cup\Gamma_2$ with $\Gamma_1=\{z\in \Bbb C, |z|=r\}$ and $\Gamma_2=\{z\in \Bbb C, |z|=1\}.$   We have seen that   $F(\Gamma_i)$ is a convex curve contained in an open hemisphere, $i=1,2.$  Along $\Gamma_i,$ the unit normal $N$ to $F(\mathcal A)$ is a unit normal to $F(\Gamma_i)$ in $\Bbb S^2.$ As $F(\Gamma_i)$ is convex, we infer that $N$ is injective along $\Gamma_i.$ So  $N(\Gamma_i)$ separates $\Bbb S^2$ into 2 topological disks. Since  $N$ is a local diffeomorphism,  there is a thin strip in $\mathcal A$ around $\Gamma_i$ on which $N$ is injective and its image by $N$ lies on the same component of $\Bbb S^2\setminus N(\Gamma_i).$ We denote by $D_i$ the component of $\Bbb S^2\setminus N(\Gamma_i)$ which does not contain this image. 

We let  now $\widetilde{\mathcal A}$ denote the (topological) sphere obtained by attaching  a copy of $D_i$ to $\mathcal A$ using the diffeomorphism
  $F |_{\Gamma_i}: \Gamma_i\to \partial D_i,$ for  $i=1,2.$  (see, for instance, Theorem 9.29 in \cite{lee} for the details of such a construction). We define 
$\widetilde N: \widetilde{\mathcal A} \to \Bbb S^2$ by
\begin{equation*}
\widetilde N =\begin{cases} N\qquad &{\text{on}}\quad \mathcal A\\
\text{identity }\quad &{\text{on}}\quad D_i, i=1,2.
\end{cases}
\end{equation*}

Note that $\widetilde N$ is well defined since, by construction of   $\widetilde{\mathcal A},$ each  $p\in \Gamma_i$ is identified with $F(p)\in\partial D_i$.  Observe now that $\widetilde N$ is a local homeomorphism. This is clear for points not on $\partial D_i, i=1,2,$ and for points on $\partial D_i, i=1,2,$ this follows from the choice of the $D_i$'s. So $\widetilde N$  is a covering map  and  is therefore
 a global homeomorphism onto $\Bbb S^2 .$ This shows $N$ is a diffeomorphism onto $N(\mathcal A).$ 

We will now show $F$ is an embedding. First, we consider the support function of the immersion, $v=\langle F,N\rangle,$ defined on $ \mathcal A.$ It is a well known fact that $v$ is a Jacobi function, that is, it satisfies the equation: $\Delta_{\mathcal A} v -2K v=0,$ where $\Delta_{\mathcal A}$ is the Laplacian on $\mathcal A$ of the metric induced by the immersion $F$ in $\Bbb R^3.$ Also, by  the free boundary condition $v=0$ on $\partial\mathcal A.$ Now, as $N$ is a diffeomorphism onto $\Omega:=N(\mathcal A)\subset \Bbb S^2,$ 
the  function $ u := v\circ N^{-1}$ verifies: 
\begin{align*} \Delta u +2u &=0\qquad \text{in}\quad \Omega\\
u&=0 \qquad \text{on}\quad \partial\Omega,
\end{align*}
where $\Delta$ is the Laplacian on $\Bbb S^2.$

It follows from Lemma \ref{sign} that $u$ and hence  the support  function $v=\langle F, N\rangle$ does not vanish in the interior of $\mathcal A.$ Up to changing the orientation of $\mathcal A,$ we may assume $v>0$ in int$(\mathcal A).$ 

We now consider, as in \cite{RV},  the radial projection on $\Bbb S^2,$ that is, the mapping:
$X: \mathcal A \longrightarrow \Bbb S^2$ defined by 
$$ X(p)= \frac{F(p)}{|F(p)|},\quad p\in \mathcal A.$$
$F$ is well defined as $v$ has no zeros on int$(\mathcal A)$ and $X=F$ on $\partial \mathcal A.$ We now check that $X$ is a local diffeomorphism on int$(\mathcal A).$ Let $p\in$
 int$( \mathcal A)$ and 
$w\in T_p\Sigma,$ we have 

\begin{equation}\label{derivative}  dX_p(w)= \frac{dF_p(w)}{|F(p)|} - \frac{\langle dF_p(w),F(p)\rangle}{|F(p)|^3} F(p).
\end{equation}

If $dX_p(w)=0,$ we deduce that 
$$0= \frac{\langle dF_p(w), N(p)\rangle}{|F(p)|} =\frac{\langle dF_p(w),F(p)\rangle}{|F(p)|^3} \langle  F(p), N(p)\rangle.$$
As $\langle  F(p), N(p)\rangle\neq 0,$ we have $\langle dF_p(w),F(p)\rangle=0.$ Going back to (\ref{derivative}),
we get $dF_p(w)=0$ and thus $w=0.$ This shows $X$ is a local diffeomorphism on int$(\mathcal A).$ 

We now analyze the behaviour of $F$ at the boundary points. Denote by $\nu$ the exterior unit normal field to $\partial\mathcal A$ in $\mathcal A$ for the metric induced by the immersion. Let $p\in\partial\mathcal A$ and $\gamma: (-\epsilon,0]\to \mathcal A$ an arclength  parametrized curve satisfying $\gamma(0)=p,\, 
\dot\gamma (0)=\nu(p)$ and set $\overline\gamma=X\circ\gamma.$ We have $\langle X(\gamma(0)),N(p)\rangle=\langle F(p),N(p)\rangle=0$ and using (\ref{derivative}), one checks directly that
\begin{align*}
\frac{d}{dt}|_{0} \langle X(\gamma(t)),N(p)\rangle &=0\\
\frac{d^2}{dt^2}|_{0} \langle X(\gamma(t)),N(p)\rangle &= \langle \ddot{\overline{\gamma}}(0), N(p)\rangle = \text{II}(\nu(p),\nu(p))
\end{align*}
As $\nu$ is a principal direction, we have $|\text{II}(\nu,\nu)|= \sqrt{-K}\neq 0.$ Therefore,  $\text{II}(\nu,\nu)$ has a sign along 
$\Gamma_i,\, i=1,2.$ So for $t$ small, $X(\gamma(t))$ lies in the component of $\Bbb S^2\setminus F(\Gamma_i)$ which has $N(p)$ 
as outwards (resp. inwards) pointing normal at $F(p)$ if $\text{II}(\nu,\nu)>0$ (resp. if $\text{II}(\nu,\nu)<0).$ It follows that there is a thin strip in the interior of $\mathcal A$ around $\Gamma_i$ which projects 
on this component. We call $D_i$ the component of  $\Bbb S^2\setminus F(\Gamma_i)$ which does not intersect this projection.

 As before, we let
 $\widetilde{\mathcal A}$ denote the (topological) sphere obtained by attaching a copy of $D_i$ to $\mathcal A$ using the diffeomorphism $F|_{\Gamma_i}: \Gamma_i\to \partial D_i,$ for $i=1,2$   (again, one can see  Theorem 9.29 in \cite{lee} for the details of such a construction). 

Define 
$\widetilde X: \widetilde{\mathcal A} \to \Bbb S^2$ by
\begin{equation*}
\widetilde X =\begin{cases} X\qquad &{\text{on}}\quad \mathcal A\\
\text{identity }\quad &{\text{on}}\quad D_i, i=1,2.
\end{cases}
\end{equation*}

Note that $\widetilde X$ is well defined since, by construction of   $\widetilde {\mathcal A},$ each  $p\in \Gamma_i$ is identified with $X(p)=F(p)\in\partial D_i$. By its very definition and the choice of the $D_i$'s,   $\widetilde X$ is a local homeomorphism. Therefore
$\widetilde X$ is a covering of $\Bbb S^2$ and thus a global homeomorphism. If follows that $F$ is an embedding and we conclude by  uniqueness of the critical catenoid \cite{CG, PZ}.
\end{proof}

\section{An overdetermined eigenvalue problem on the sphere}\label{sec:overdetermined}

As we indicated in the introduction,  it was observed in \cite{S}, that free boundary minimal surfaces of genus zero in $\Bbb B^3$ are closely related to an overdetermined eigenvalue problem on domains of $\Bbb S^2.$ More precisely, let $\Omega\subsetneq \Bbb S^2$ be a smooth domain and suppose there exists a smooth function $u$ on $\Omega$ verifying 
\begin{equation}\label{two}
\Delta u+ 2u = 0 \quad \text{in} \quad \Omega
\end{equation}

Then, if $u$ is not the restriction of a linear function  on $\Bbb R^3$ to $\Omega,$ the map
\begin{align*} X_u&: \Omega\longrightarrow \Bbb R^3\\
X_u(p)&=\nabla u (p) + u(p).p
\end{align*} 
is  not constant, harmonic and almost conformal and so defines a branched minimal surface (cf. \cite{R}). Here $\nabla u$ denotes the gradient of $u$ on $\Omega$ viewed as a vector in $\Bbb R^3.$  Moreover, the Gauss map of $X_u$ is the identity: $N(p)=p,$ for all $p\in\Omega.$ 

Suppose that, in addition, $u$ is smooth up to the boundary and satisfies the following boundary conditions:
\begin{equation}\label{boundary}
u=0  \quad \text{and} \quad
|\nabla u|=1 \quad \text{on}\quad \partial\Omega.
\end{equation}

The key observation (cf. \cite{S}) is  that $X_u$ defines a branched minimal surface with free boundary in $\Bbb B^3.$ Indeed,
for $p\in\partial\Omega,$ we have  
$$
|X_u(p)|^2 =|\nabla u(p)|^2+u^2(p)=1\,\text{ and} \, \langle X_u(p),N(p)\rangle =\langle \nabla u(p),p\rangle=0$$
and by the maximum principle $X_u(\text{int}(\Omega)) \subset \text{int}(\Bbb B^3).$
Conversely, if $\Sigma$ is a free boundary minimal surface in $\Bbb B^3$ with injective Gauss map, then its image by the Gauss map in 
$\Bbb S^2$ supports a smooth function verifying (\ref{two}) and (\ref{boundary}), namely the support function of the immersion parameterized by the inverse of the Gauss map. Applying this to the critical catenoid, we obtain the rotational annulus 
 $$ \{(x,y,z)\in \Bbb S^2,\, |z|\leq 1/L\},$$
  where $L$ is the unique positive solution to the equation $L=\coth L$ and 
 a function solving the problem is given by the restriction of the function 
 $$ \frac{1}{L\coth L}\left( 1- z\operatorname{arctanh} z \right).$$
 
 We will now prove this is the only such domain  of annular type, up to rotations  :
 %%%%%%%%%%%%%%%%%%%%%%%
\begin{theorem}\label{annulus} Let $\Omega$ be a $\mathcal C^{2,\alpha}$ domain in $\Bbb S^2$ having the topology of an annulus. Suppose there exists a function $u\in \mathcal C^2(\Omega)\cap \mathcal C^1(\overline \Omega)$  solution to  the overdetermined  problem:

\begin{equation}\label{problem}
\begin{cases} \Delta u+ 2u = 0 \quad \text{in} \quad \Omega\\
\qquad \quad u=0  \quad \text{on}\quad\partial\Omega\\
\quad\,\,\,\, |\nabla u|=1 \quad \text{on}\quad \partial\Omega.
\end{cases}
\end{equation}

Then $\Omega$ is, up to rotation, the rotational annulus
 $$  \{(x,y,z)\in \Bbb S^2,\, |z|\leq 1/L\},$$
  where $L$ is the unique positive solution to the equation $L=\coth L$ and 
  $u$ is, up to sign,  the restriction of the function  
 $$ \frac{1}{L\coth L}\left( 1- z\operatorname{arctanh} z \right).$$
   \end{theorem}
   \medskip
\begin{proof} It follows from the regularity results in \cite{KN} that $\partial\Omega$ is analytic and by elliptic regularity $u$ is analytic in
 $\overline\Omega.$
Clearly $u$ is not the restriction to $\Omega$ of a linear map, so, as we recalled above, the map
\begin{align*} X_u&: \Omega\longrightarrow \Bbb R^3\\
X_u(p)&=\nabla u (p) + u(p).p
\end{align*} 
defines a branched minimal annulus with free boundary in $\Bbb B^3$ with the identity as Gauss map.  The argument in the proof of Theorem \ref{embedding}
shows that $X_u$ is free of branch points and  its Gauss curvature does not vanish. We will show that $X_u$ is an embedding on each of its boundary components. Let $\gamma$ be a parametrization  of a component of $\partial\Omega$  and set $\overline\gamma=X_u\circ\gamma.$  Again, by the argument in the proof of Theorem \ref{embedding},
 the geodesic curvature of  $\overline\gamma$ in $\Bbb S^2$ verifies  $|\kappa_g|=\sqrt{-K}\neq 0.$ So $\overline\gamma$   is a locally convex curve in $\Bbb S^2.$ To see that $\overline\gamma$ is embedded, observe that  its polar curve  in $\Bbb S^2,$  which  is the curve $N\circ\gamma=\gamma,$ is also locally convex (its geodesic curvature is $1/\kappa_g,$ see for instance \cite{BD}).   So $\gamma$ is an embedded  locally convex curve in $\Bbb S^2,$    it is therefore contained in an open hemisphere and is globally convex (cf. \cite{F}). It follows, in turn, that  
its  polar curve, which  is $\overline\gamma,$ is also embedded.

We can thus  apply Theorem \ref{embedding} to conclude that $X_u(\Omega)$ is, up to rotation, the critical catenoid and $\Omega$ 
is its image by the Gauss map. Moreover  $u$ is unique up to sign since $\lambda^D_1(\Omega)=2$ (cf. Lemma \ref{sign}),. 
\end{proof}
\medskip
\begin{remark} The correspondence, through the support function,  between solutions to the problem (\ref{problem}) and (branched) free boundary minimal surfaces of genus zero in $\Bbb B^3$ suggested a possible way to prove the uniqueness of the critical catenoid by proving first Theorem \ref{annulus}. An attempt using this approach  was conducted by Espinar and Marin \cite{EM} and lead to a proof of the uniqueness of the critical catenoid under an additional hypothesis.
\end{remark}
\bigskip
In their recent work, Hines, Colesar and McGrath \cite{HCM} proved the existence, for any large enough integer $m,$ of a domain in $\Bbb S^2$  having $m$ boundary components and supporting a solution to the overdetermined eigenvalue problem (\ref{problem}). This gives rise to free boundary branched minimal surfaces in $\Bbb B^3$  and  naturally raises the question  whether  some of these examples can be free of branch points. It turns out the answer is negative. Indeed, we have the following 
\begin{theorem} 
Let $\Sigma$ be an  immersed minimal surface with free boundary in $\Bbb B^3.$ Suppose its Gauss map is injective. Then $\Sigma$ is, up to a rotation,   the critical catenoid.
\end{theorem}

\begin{proof} By assumption $\Sigma$ has genus zero. Call $r$  the  number of its boundary components and let $\gamma$ be a parametrization by arclength of a boundary component.  As $\Sigma$ has free boundary in 
$\Bbb B^3,$ the position vector coincides with the exterior unit normal to $\partial\Sigma$ in $\Sigma.$ So the geodesic curvature of $\gamma$ in $\Sigma$ verifies: $\kappa_g=\langle \ddot\gamma,-\gamma\rangle= \langle\dot\gamma,\dot\gamma\rangle=1.$
Denoting by $|\partial\Sigma|$ the length of $\partial\Sigma,$ we have by Gauss-Bonnet formula:
$$\int_{\Sigma} K + |\partial\Sigma| =2\pi(2-r),$$
$K$ being the Gauss curvature of $\Sigma.$ As $N$ is injective, $\int_{\Sigma} K>-4\pi$ and by a result in \cite{RV} 
$$|\partial\Sigma|\geq 2\pi.$$
It follows that 
$$ r<3$$
Thus $r=1$ or $2.$ By Nitsche's Theorem \cite{N}, if $r=1,\,\Sigma$ must be an  equatorial  disk and the Gauss map is then a constant.
Hence  $\Sigma$ is an annulus with injective Gauss map. Parameterizing $\Sigma$ by the inverse of its Gauss map, we get an annulus 
in $\Bbb S^2$ verifying the hypotheses of Theorem \ref{annulus} and therefore $\Sigma$ coincides, up to a rotation, with the critical catenoid.

\end{proof}

 \end{document}